\documentclass[leqno,a4paper,oneside]{amsart}
\usepackage{amsmath,amssymb,amscd,mathrsfs}
\usepackage{graphicx,subfigure,caption,comment,multirow,color}
\usepackage{mathptmx,courier}
\usepackage[scaled=1.0]{helvet}

\title{The search for non-symmetric ribbon knots}

\author{Christoph Lamm}

\theoremstyle{plain}
  \newtheorem{theorem}{Theorem}[section]
  
  \newtheorem{proposition}[theorem]{Proposition}

\theoremstyle{definition}
  \newtheorem{definition}[theorem]{Definition}
  \newtheorem{question}[theorem]{Question}
  
  \newtheorem{example}[theorem]{Example}
	\newtheorem{conjecture}[theorem]{Conjecture}

\newcommand{\R}{\mathbb{R}}

\begin{document} %%%%%%%%%%%%%%%%%%%%%%%%%%%%%%%%%%%%%%%%%%%%%%%%%%%%%%%%%%%%
%%%%%%%%%%%%%%%%%%%%%%%%%%%%%%%%%%%%%%%%%%%%%%%%%%%%%%%%%%%%%%%%%%%%%%%%%%%%%

\begin{abstract}
We present the results of Axel Seeliger's tabulation of symmetric union presentations 
for ribbon knots with crossing numbers 11 and 12 and exhibit possible examples for ribbon 
knots which are not representable as symmetric unions. In addition, we give a complete 
atlas of band diagrams for prime ribbon knots with 11 and 12 crossings.
\end{abstract}


%%%%%%%%%%%%%%%%%%%%%%%%%%%%%%%%%%%%%%%%%%%%%%%%%%%%%%%

% captions seem to be too near to the tables
\captionsetup{belowskip=10pt,aboveskip=10pt}

% scaling of diagrams in the main overview (4 pages)
\def\scaling{0.8}

% box colour for determinants
\definecolor{myboxcolour}{gray}{0.8}

% use this if 'oneside' is active
\reversemarginpar

%%%%%%%%%%%%%%%%%%%%%%%%%%%%%%%%%%%%%%%%%%%%%%%%%%%%%%%

\maketitle

\section{Introduction and outline of results} \label{sec:Introduction}

\subsection{The setting}
We recall that a knot $K \subset \R^3$ is a \emph{ribbon knot} if it bounds a
smoothly immersed disk $\mathbb{D}^2 \looparrowright \R^3$ whose only
singularities are ribbon singularities, two sheets intersecting in an arc whose
preimage consists of a properly embedded arc in $\mathbb{D}^2$ and an
embedded arc interior to $\mathbb{D}^2$.
Ribbon knots are frequently illustrated using the connected sum of a knot $K$ and
its mirror image with reversed orientation, $-K$. The notion of symmetric union, introduced
in 1957 by Shin'ichi Kinoshita and Hidetaka Terasaka \cite{KinoshitaTerasaka}, generalises this by
inserting crossings on the symmetry axis. The construction of a band with ribbon
singularities remains the same as for $K\#-K$, with additional half-twists on the axis
for each inserted crossing, see Figure \ref{fig:achter_3D}.
This article analyses the question which ribbon knots can be represented by symmetric
unions -- we call them symmetric ribbon knots -- and which are non-symmetric.
Currently the question whether non-symmetric ribbon knots exist is open.

\begin{figure}[hbtp]
  \centering
  \includegraphics[scale=0.62]{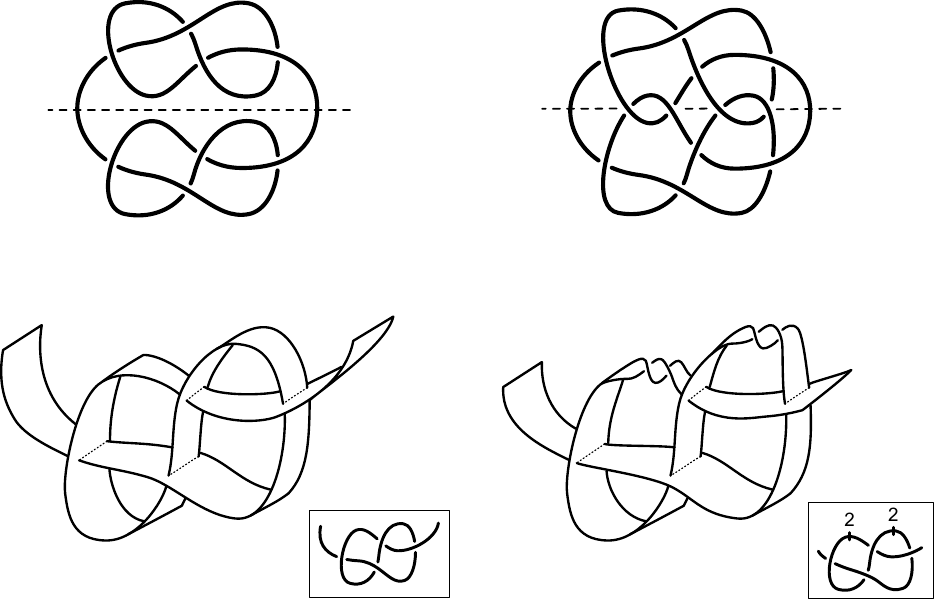}
  \caption{Symmetric unions and ribbons for the Kanenobu family of knots \cite{Kanenobu}, based on
	a diagram of $4_1\#-4_1$. On the right, the knot 12n838 is obtained by inserting crossings on the
	axis, in this case resulting in two full twists of the band. The insets show an abbreviated notation.}
  \label{fig:achter_3D}
\end{figure} 

\subsection{Knot tabulation and unknown values in the table of knots}
The tabulation of prime knots to 16 crossings was completed 20 years ago \cite{HosteThistlethwaiteWeeks}.
A systematic list of unknown values for knot invariants was initiated in 2005 by Jae Choon Cha and Charles Livingston 
\cite{ChaLivingstonUnknown}. Input for ribbon knots came in 2006 from Jenelle McAtee, Alexander Stoimenov 
and the author and for prime knots with 11 and 12 crossings 19 cases remained open. In 2008 Chris Herald, 
Paul Kirk and Charles Livingston ruled out 16 of these and showed that 12a990 is smoothly slice 
\cite{HeraldKirkLivingston}; the two cases not yet settled were 11n34 and 12a631.

\subsection{The search for symmetric unions}
In 2006 symmetric union presentations were known for all 21 prime ribbon knots with 10 or fewer crossings and for all 
2-bridge ribbon knots \cite{Lamm}, \cite{Lamm_talk}. In the article \cite{EisermannLamm2007} we proposed a search project
in order to assess how many ribbon knots with crossing number $\ge 11$ are symmetric unions.
This search was performed in 2012 by Axel Seeliger \cite{Seeliger} for the range of crossing numbers 11 
through 14 and unexpectedly produced a symmetric ribbon for 12a631. As a result, with the possible exception 
of 11n34, the list of prime ribbon (and smoothly slice) knots up to 12 crossings was complete.
Recently, in 2018, the knot 11n34 was shown not to be smoothly slice \cite{Piccirillo}.
Seeliger's search showed that out of 137 prime ribbon knots with 11 and 12 crossings the surprisingly 
high number of 122 knots have symmetric union presentations. In this article the diagrams of all
of these knots are shown in an efficient way so that only 28 diagram templates are necessary.

\subsection{Non-symmetric ribbons}
The usefulness of this list of diagrams is further enhanced by four non-symmetric ribbons 
which generate all prime ribbon knots for which a symmetric diagram was not found.
These four non-symmetric ribbons are shown in Figures \ref{fig:ribbonTransformation}, 
\ref{fig:ribbon11a103} and \ref{fig:ribbons_C_D}.
They generate knots with determinants 9, 81 and 225 and we discuss their properties with the goal to find 
candidates for prime or composite ribbon knots which do not allow a symmetric union presentation. 
The most promising candidates seem to be the composite ribbon knots $3_1\#8_{10}$ and $3_1\#8_{11}$.

The 28 diagram templates and 4 non-symmetric ribbons constitute a complete atlas of knot diagrams
for ribbon knots with 11 and 12 crossings, showing the ribbon property. For many knots it is also
the first time that ribbon diagrams are available at all: for them the information to be ribbon was 
the result of a computer search, entered into the database Knotinfo \cite{ChaLivingstonKnotInfo}
without publishing the diagram information.

\subsection{Characterisation and equivalences of bands}
In the last two sections, we characterise bands constructed from symmetric union diagrams 
and introduce the knotoid notation for symmetric ribbon disks.

We distinguish two equivalence relations for ribbon knots: a knot given by a specific ribbon can be transformed 
to a symmetric union in a way that the ribbon is respected, or to a symmetric union presentation with any ribbon. 
We give examples and conjectures containing both types of equivalences. However, this topic -- requiring the
definition of band moves, similar to Reidemeister moves for knots -- is not pursued in depth.

\section{Background on slice knots and band-sums} \label{sec:Background}
For the purpose of this article we mainly need the following property of symmetric unions:
a knot with a symmetric union diagram obtained from $K\#-K$ is a ribbon knot with determinant 
equal to the square of the determinant of $K$. The ingredient $K$ is called the \textsl{partial knot} 
of the symmetric union diagram. As we usually do not need to distinguish $K$ and $-K$ as partial knots, 
we use the notation $K_{\pm}$. Because of the important property of the determinants, our atlas of 
symmetric ribbon knots is organised by (square) determinants.
We refer to the articles \cite{EisermannLamm2007} and \cite{EisermannLamm2011} for additional topics, for 
instance symmetrical equivalence and the two-variable Jones polynomial defined for symmetric union diagrams.

\subsection{Three equivalent definitions of ribbon knots}
We defined ribbon knots as immersed disks with only ribbon singularities.
Alternatively, ribbon knots may be defined using band-sums, a generalisation of the connected sum of knots,
see \cite{Kanenobu2010} for more details. To motivate the occurence of glueing in bands we need a digression 
on the definition of slice knots in the 4-dimensional setting. 

\begin{figure}[hbtp]
  \centering
  \includegraphics[scale=0.85]{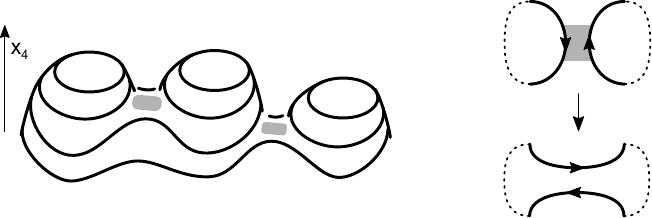}
  \caption{A disk with three local maxima and two saddles and the band-sum definition}
  \label{fig:band-sum}
\end{figure}

Recall that a knot is smoothly slice if it is the boundary of a smooth disk $\mathbb{D}^2$ embedded 
in $\R^4_+ = \{ x \in  \R^4 \mid x_4 \ge 0 \}$. If the disk is given in Morse position, it can be described 
as a movie. Local maxima correspond to new trivial components, saddles to the insertion of bands 
(band-sums) and local minima to the disappearance of trivial components. Ribbon knots are characterised 
by disks without local minima. The slice-ribbon problem asks if every slice knot is a ribbon knot.

A band-sum in which the number of components of a link decreases by one (this 
is also called a `fusion')  is illustrated in the right of Figure \ref{fig:band-sum}. 
\begin{figure}[hbtp]
  \centering
  \includegraphics[scale=0.65]{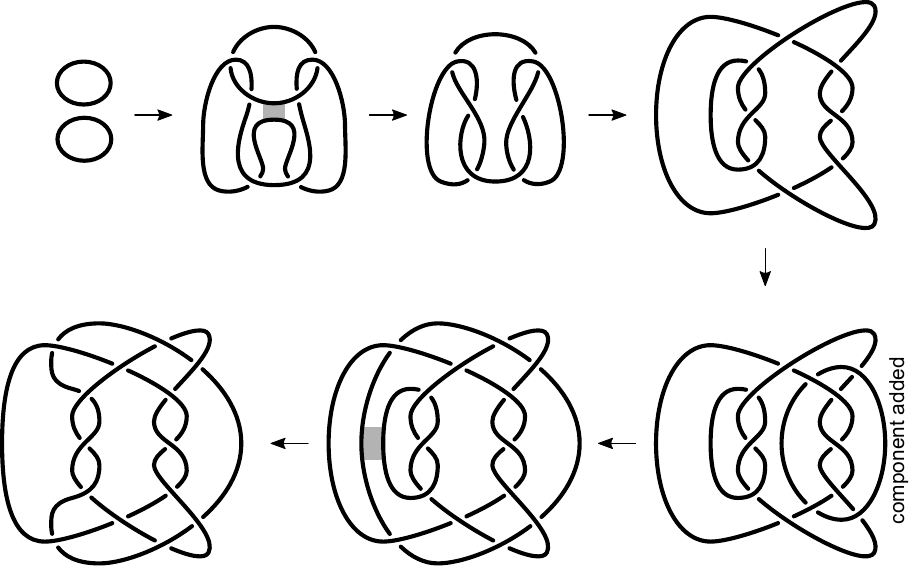}
  \caption{Example of a movie (for the knot 12a1225)}
  \label{fig:saddle_transformation_12a1225}
\end{figure}

The insertion of bands to connect different components of a link occurs therefore in a natural way and 
by the arguments of the previous paragraph ribbon knots can be defined as knots obtained from a trivial 
$m+1$-component link by doing $m$ band-sum operations for some integer $m \ge 1$. In each operation the 
number of components decreases by one so that the result is a knot. As an example for $m=2$ we show a 
movie for the knot 12a1225 in Figure \ref{fig:saddle_transformation_12a1225}. It is based on a diagram in 
\cite{Watson}. The two areas affected by band-sums are marked in grey. The (deformed) second band can 
be clearly seen in the diagram with ribbon singularities in the left of Figure \ref{fig:ribbons_C_D}.

\subsection{Symmetric ribbon knots defined as symmetric band-sums of trivial links}
We would like to apply the definition by band-sums to \textsl{symmetric} ribbon knots: 
in this case the trivial $m+1$-component link and the bands need to be symmetric to a plane;
for the inserted bands we allow twists as in the definition of symmetric union diagrams.
This definition with band-sums of a trivial link is equivalent to the usual definition of symmetric unions because
the symmetric band can be cut open at the singularities and at the twist places, yielding a trivial link, and conversely 
bands can be glued in at the same places to obtain the symmetric band, see Figure \ref{fig:achter_3D_split}. 
This characterisation was also used by Toshifumi Tanaka in \cite{Tanaka2010}.

\begin{figure}[hbtp]
  \centering
  \includegraphics[scale=0.8]{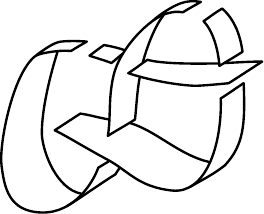}
  \caption{A trivial 6-component link which can be transformed into 12n838 by band-sums (compare with Figure 
	\ref{fig:achter_3D}). Untwisted bands are glued in where singularities have been cut out and twisted bands
	in the twist places.}
  \label{fig:achter_3D_split}
\end{figure}

We now analyse the symmetry in the movie in Figure \ref{fig:saddle_transformation_12a1225}.
The first band-sum is applied to a symmetric trivial 2-component diagram and results in a symmetric union
diagram for $3_1\#-3_1$. The third trivial component is added symmetrically with respect to a new
axis perpendicular to the previous one and the second band-sum is again done symmetrically but 
with respect to the second axis. Of course, the fact that two different symmetry axes are used in
the movie does not prove that the knot 12a1225 is a non-symmetric ribbon knot; it might possess
a symmetric presentation as well.

A symmetry condition for bands with ribbon singularities due to Paolo Aceto is given in Section \ref{sec:banddiagrams}.

\section{Tabulation of symmetric ribbon knots with 11 and 12 crossings} \label{sec:Tabulation}

\subsection{Seeliger's tabulation method}
In a computer search Seeliger used the Dowker-Thistle\-thwaite code to generate partial knots $K$, applying a 
new graph-theoretic method for testing realisability \cite{Seeliger}. He then generated diagrams for $K \# -K$, 
allowing additional variants in cases where the Dowker-Thistlethwaite code is not unique (composite partial knots).
After adding twist tangles on the symmetry axis, he checked the Jones polynomial of the symmetric
union against the list of all Jones polynomials of prime ribbon knot candidates up to 14 crossings (defined as 
knots which satisfy the $f(t)\cdot f(1/t)$ condition for the Alexander polynomial and having zero signature).
Identification was done with Knotscape \cite{Knotscape}. This resulted in a complete search for symmetric 
diagrams up to 16 crossings. Diagrams with more than 16 crossings were also generated, using heuristical strategies 
in order to find additional knots with low crossing number. 
Minimality of the crossing number is therefore assured also for 17 crossings.
Once a knot had been found it was omitted from the search -- we do not know if another
diagram with the same number of crossings but e.g.\,different partial knot would have been found. 
For 11 and 12 crossing knots the result is summarised in the proposition:

\enlargethispage{1cm}

\begin{proposition}[Seeliger]
Out of 137 prime ribbon knots with 11 and 12 crossing 122 knots have symmetric union presentations.
\end{proposition}

The `symmetry status' of the remaining 15 prime ribbon knots is currently unknown. We discuss their properties
in Section \ref{sec:Ribbons}. 

%We note that it is still not known whether the Conway knot 11n34, a mutant of the
%Kinoshita-Terasaka knot 11n42, is smoothly slice and we exclude this knot from the discussion.

\small
\noindent
\begin{table}[htbp]
\begin{tabular}{|lrl|lrl|lrl|lrl|}
\hline
11a28   & 121	&	m	& 12a100	& 225	&	sym	&	12a1029	& 81	&	m	&	12n312	& 49	&	m	\\
11a35   & 121	&	m	& 12a173	& 169	&	m	&	12a1034	& 121	&	m	&	12n313	& 1	&	m	\\
11a36   & 121	&	m	& 12a183	& 121	&	m	&	12a1083	& 169	&	sym	&	12n318	& 1	&	m	\\
11a58   & 81	&	sym	& 12a189	& 225	&	sym	&	12a1087	& 225	&	sym	&	12n360	& 49	&	m	\\
11a87   & 121	&	m	& 12a211	& 169	&	m	&	12a1105	& 289	&	m	&	12n380	& 81	&	m	\\
11a96   & 121	&	m	& 12a221	& 169	&	m	&	12a1119	& 169	&	m	&	12n393	& 49	&	m	\\
11a103	& 81	&	--	& 12a245	& 225	&	sym	&	12a1202	& 169	&	m	&	12n394	& 25	&	m	\\
11a115	& 121	&	m	& 12a258	& 169	&	m	&	12a1225	& 225	&	--	&	12n397	& 49	&	m	\\
11a164	& 169	&	sym	& 12a279	& 169	&	m	&	12a1269	& 169	&	m	&	12n399	& 81	&	m	\\
11a165	& 81	&	--	& 12a348	& 225	&	--	&	12a1277	& 121	&	m	&	12n414	& 25	&	m	\\
11a169	& 121	&	m	& 12a377	& 225	&	sym	&	12a1283	& 81	&	m	&	12n420	& 81	&	m	\\
11a201	& 81	&	--	& 12a425	& 81	&	m	&	      	&	       &		     &	12n430	& 1	&	m	\\
11a316	& 121	&	m	& 12a427	& 225	&	m	&	12n4	& 81	&	m	&	12n440	& 81	&	m	\\
11a326	& 169	&	sym	& 12a435	& 225	&	sym	&	12n19	& 1	&	m	&	12n462	& 25	&	m	\\
            &	      &	   	& 12a447	& 121	&	m	&	12n23	& 9	&	m	&	12n501	& 49	&	m	\\
11n4    & 49	&	m	& 12a456	& 225	&	sym	&	12n24	& 49	&	m	&	12n504	& 121	&	m	\\
11n21   & 49	&	m	& 12a458	& 289	&	sym	&	12n43	& 81	&	m	&	12n553	& 81	&	m	\\
11n37   & 25	&	m	& 12a464	& 225	&	sym	&	12n48	& 49	&	m	&	12n556	& 81	&	m	\\
11n39   & 25	&	m	& 12a473	& 289	&	sym	&	12n49	& 81	&	sym	&	12n582	& 9	&	m	\\
11n42   & 1	&	m	& 12a477	& 169	&	m	&	12n51	& 9	&	--	&	12n605	& 9	&	m	\\
11n49   & 1	&	m	& 12a484	& 289	&	sym	&	12n56	& 9	&	--	&	12n636	& 81	&	m	\\
11n50   & 25	&	m	& 12a606	& 169	&	m	&	12n57	& 9	&	--	&	12n657	& 81	&	m	\\
11n67   & 9	&	--	& 12a631	& 225	&	sym	&	12n62	& 81	&	--	&	12n670	& 25	&	m	\\
11n73   & 9	&	--	& 12a646	& 169	&	m	&	12n66	& 81	&	--	&	12n676	& 9	&	sym	\\
11n74   & 9	&	--	& 12a667	& 121	&	m	&	12n87	& 49	&	m	&	12n702	& 121	&	m	\\
11n83   & 49	&	m	& 12a715	& 169	&	m	&	12n106	& 81	&	m	&	12n706	& 49	&	m	\\
11n97   & 9	&	--	& 12a786	& 169	&	sym	&	12n145	& 25	&	m	&	12n708	& 49	&	m	\\
11n116  & 1	&	m	& 12a819	& 169	&	m	&	12n170	& 81	&	m	&	12n721	& 25	&	m	\\
11n132  & 25	&	m	& 12a879	& 121	&	m	&	12n214	& 1	&	m	&	12n768	& 25	&	m	\\
11n139  & 9	&	m	& 12a887	& 289	&	sym	&	12n256	& 25	&	m	&	12n782	& 81	&	m	\\
11n172  & 49	&	m	& 12a975	& 225	&	sym	&	12n257	& 25	&	m	&	12n802	& 121	&	m	\\
             &	    &		& 12a979	& 225	&	sym	&	12n268	& 9	&	m	&	12n817	& 49	&	m	\\
12a3     & 169	&	sym	& 12a990	& 225	&	--	&	12n279	& 25	&	m	&	12n838	& 25	&	m	\\
12a54   & 169	&	m	& 12a1011	& 121	&	m	&	12n288	& 49	&	m	&	12n870	& 25	&	m	\\
12a77   & 225	&	sym	& 12a1019	& 361	&	m	&	12n309	& 1  	&	m	&	12n876	& 81	&	m	\\
\hline 
\end{tabular}
%\bigskip
\caption{A list of all 137 prime ribbon knots with crossing numbers 11 or 12 together with 
determinant and status information. Abbreviations: `m' = a minimal symmetric union diagram is known, 
`sym' = a symmetric union diagram is known but is not necessarily minimal, `--' = no symmetric 
union diagram has been found.}
\label{tab:RibbonKnotList}
\end{table}

\normalsize

\subsection{Finding a small number of diagram templates}
To save space in the appendix of this article, we tried to describe as many of the 122 diagrams as possible in 
a template form. For instance, all 15 ribbon knots with partial knot $5_2$ can be described in a uniform way
using a template with 3 twist parameters. In other cases several templates were needed and they are listed as
`version a', `version b' (for the trivial partial knot and for $6_1$, $6_2$, $7_3$, $7_5$) and, in 
addition, `version c' for $3_1$. We show the templates with one choice of twist parameters and 
underline the corresponding knot in the annotating list. If a diagram is used only for one ribbon knot, 
then we do not depict it as a template with parameters but as an individual symmetric union diagram. 
See, for example, the diagram for 12n706 with determinant 49 in the appendix.
The template parameters were checked with Knotscape to make sure that the knot types are correct.
Note that partial knots and the number of inserted twists on the axis agree with Seeliger's list, but our
diagrams need not be the same and may be symmetrically inequivalent. 
\enlargethispage{0.5cm}

An overview is given in Table \ref{tab:RibbonKnotList}. It contains all 137 prime ribbon knots 
with 11 and 12 crossings, their determinant and a `symmetry status' information. 

As an example, assume that you are interested in alternating ribbon knots with 11 crossings and determinant 169. 
In the first column of Table \ref{tab:RibbonKnotList} we find the two knots 11a164 and 11a326 with the 
information `sym' meaning that for these knots a symmetric but not necessarily minimal diagram exists.  
They are found in the appendix for the $7_3$ `version a'-template and for both knots
5 crossings are inserted on the axis, giving 19 crossings in total. Because the search was complete only
through 16 crossings, the diagrams need not be minimal.

As a second example, we look at all non-alternating knots in Table \ref{tab:RibbonKnotList} and observe that with 
the exception of two knots (12n49 and 12n676) all symmetric diagrams are minimal (they have 17 or less crossings).
This is in contrast to \textsl{alternating} knots where many symmetric diagrams with 18 or more crossings occur.

\subsection{Properties of the symmetric union presentations in the atlas}
To give some orientation, we comment on the properties of the partial knots, the occurence 
of composite ribbon knots and of strongly-positive-amphichiral knots in the atlas:

\smallskip
A) The following properties have been observed with respect to the partial knots:
\begin{itemize}
		\item (bridge number) Among the prime partial knots in the atlas only $8_{21}$ has bridge number 
		  greater than 2. Because all non-trivial prime knots with up to 7 crossings
			are two-bridge knots, this is not as surprising as one may think at first.
    \item (composite) There are symmetric unions with composite partial knots $3_1\#3_1$, $3_1\#-3_1$
		  and $3_1\#4_1$. These are interesting because if a symmetric union is constructed from a composite
			partial knot in another way it is itself composite.
		\item (non-alternating) There is only one non-alternating partial knot in the list -- again $8_{21}$ -- but
		  for determinants 1 and 9 we have non-alternating diagrams of the trivial knot and of the trefoil.
			All other partial diagrams in the atlas are alternating.
	\end{itemize}

\smallskip
B) We note a new phenomenon concerning composite ribbon knots: 
for crossing numbers $\le 10$ composite ribbon knots are always of the form $K\#-K$
with prime $K$, but for crossing numbers 11 and 12 other types of composite ribbon knots exist: 
$3_1\# 8_{10}$, $3_1\# 8_{11}$ and also $6_1\#3_1\#-3_1$, $3_1\#3_1\#-3_1\#-3_1$, see
\cite{Kearney}, \cite{LivingstonConcordanceGenus}.
Of course, in the cases $3_1\# 8_{10}$ and $3_1\# 8_{11}$ suitable mirror variants have to be chosen so 
that the total signature is zero. It is clear that $6_1\#3_1\#-3_1$ and $3_1\#3_1\#-3_1\#-3_1$
are symmetric because they are composed of symmetric ribbon knots. On the other hand,
$3_1\# 8_{10}$ and $3_1\# 8_{11}$ are interesting candidates for non-symmetric ribbon knots.

\smallskip
C) A knot is called strongly-positive-amphichiral if it has a diagram which is mapped to its mirror image by
a rotation of $\pi$, preserving the orientation.  

Obviously, composite knots consisting of a knot and its mirror image are strongly-positive-amphichiral. 
Prime strongly-positive-amphichiral knots are rare, however:
In \cite{BoocherDaigleHosteZheng}, p. 482, it was stated that according to \cite{HosteThistlethwaiteWeeks} there 
are only three prime knots with 12 or fewer crossings that are strongly-positive-amphichiral: $10_{99}$, $10_{123}$ 
and 12a427.
Contrary to this, Seeliger's search produced diagrams for 12a1202 and 12n706 showing this symmetry. 
It occurs for a symmetric union diagram if in addition to its symmetry axis it possesses a rotational symmetry
axis perpendicular to the plane, provided that the knot's orientation is preserved under rotation by $\pi$. 
We found such diagrams later also for 12a1019 and 12a1105 (shown in the appendix). Therefore we ask:

\enlargethispage{1cm}

\begin{question}
Is the list of prime strongly-positive-amphichiral knots with at most 12 crossings, consisting of 
$10_{99}$, $10_{123}$, 12a427, 12a1019, 12a1105, 12a1202 and 12n706, complete?
Using Knotinfo data, the only undecided case seems to be 12a435.
\end{question}

%Remark: As the only candidate the knot 12a435 remains: According to Knotinfo there are 10 prime knots 
%which are positively- or fully-amphichiral with zero signature and a square number as determinant. 
%Two of them ($8_9$ and 12a477) do not satisfy the $\Delta(t) = f(t)^2$-condition for strongly-positive-amphichiral 
%knots. We can not exclude the knot 12a435 because it shares several knot polynomial invariants with 12a427, which 
%is strongly-positive-amphichiral.

In the next section we analyse the 15 ribbon knots which were not found to be symmetric unions
and try to find common properties which may lead to a topological explanation.

\clearpage
\newpage
\section{Analysing non-symmetric ribbons} \label{sec:Ribbons}
The 15 prime ribbon knots for which a symmetric union presentation is missing have
determinants 9 (7 cases), 81 (5 cases) and 225 (3 cases). For the cases with determinants 
9 and 81 we found two non-symmetric ribbons (only two ribbons covering all 12 cases)
suggesting that these knots indeed share common topological obstructions that prevent 
their presentation as symmetric unions. The 3 cases with determinant 225 are also covered
by two ribbons.

\subsection{A non-symmetric ribbon for the remaining ribbon knots with determinant 9} \label{sub:det9}
For the 7 knots with determinant 9 (i.e.\,11n67, 11n73, 11n74, 11n97, 12n56 and 12n57) 
we found the ribbon $\mathcal{A}$ shown in Figure \ref{fig:ribbonTransformation}. 
The figure also shows an effort to symmetrise the diagram:
We put the tangle with $T$ half-twists on the symmetry axis indicated by the tangle $3_1\#-3_1$ 
in the diagram's upper part. Then the $S$-tangle appears as a `symmetry defect' on the left side. 
This indicates that for some values of $S$ it might be possible to obtain a symmetric diagram
but for others it might not be possible.

\begin{figure}[hbtp]
  \centering
  \includegraphics[scale=\scaling]{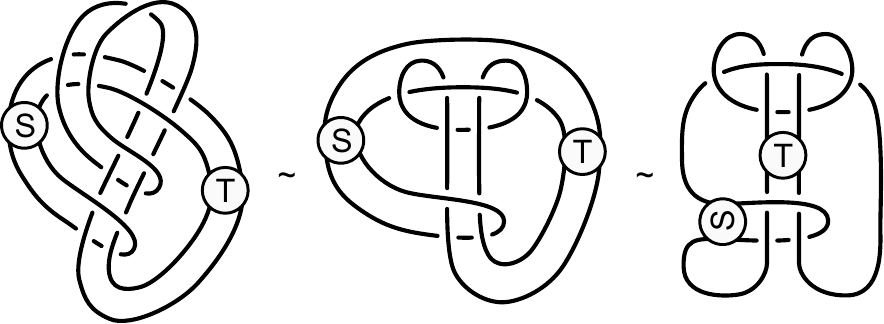}
  \caption{The ribbon $\mathcal{A}(S,T)$ and two transformations}
  \label{fig:ribbonTransformation}
\end{figure} 

Varying the twist parameters $S$ and $T$ we obtain the knots in Table \ref{fig:knotsDet9}.
For the values $S=-1,0$ they can indeed be presented as symmetric unions.

% belongs to the following table
\renewcommand{\arraystretch}{1.1}

\noindent
\begin{table}[htbp]
\begin{tabular}{|c|c|c|c|c|c|}
\hline
S $\backslash$ T &$-2$&$-1$&0&1&2 \\
\hline
$-4$ &  & 13n1666 & 13n1313 & 13n1674 & \\
$-3$ & & 12n51 & 12n56 & 12n268 & \\
$-2$ & 13n612 & 11n67 & 11n74 & 11n97 & 13n973 \\
\hline
$-1$ & $10_{140}$ & $9_{46}$ & $8_{20}$ & $6_1$ & $3_1\#-3_1$ \\
 0 & $3_1\#-3_1$ &  $6_1$ & $8_{20}$ & $9_{46}$ & $10_{140}$ \\
\hline
 1 & 13n835 & 11n97 & 11n73 & 11n67  & 13n611 \\
 2 & & 12n268 & 12n57 & 12n51 & \\
 3 & & 13n1674 & 13n1316 & 13n1666 & \\
\hline
\end{tabular}
%\bigskip
\caption{Overview of knot types for the ribbon $\mathcal{A}(S,T)$ with determinant 9 
(knots and their mirror images are not distinguished)}
\label{fig:knotsDet9}
\end{table}
% belongs to the previous table
\renewcommand{\arraystretch}{1.0}

\noindent
With the exception of 12n268, for none of the other knots in the table was a symmetric diagram found.
We discuss the knots in Table \ref{fig:knotsDet9} in more detail: 
\begin{itemize}
\item The 11 and 12 crossings knots in the table (with the exception of 12n268) are the ribbon knots with
  determinant 9 for which no symmetric union presentation was found.
\item For the knot 12n268 a symmetric diagram is known, see template $3_1$, version b.
\item For $S=0$ the ribbon can be symmetrised, as shown in Figure \ref{fig:ribbonTransformation_S0}.
  The case $S=-1$ is similar and results in a total number of half-twists on the axis of $T-2$ instead of $T+2$. 
	These knots are generated by template $3_1$, version a.
\item In the column $T=0$ the knots 11n73 and 11n74 are mutants, the same is true for 12n56 and 12n57. 
\end{itemize}

\begin{figure}[hbtp]
  \centering
  \includegraphics[scale=\scaling]{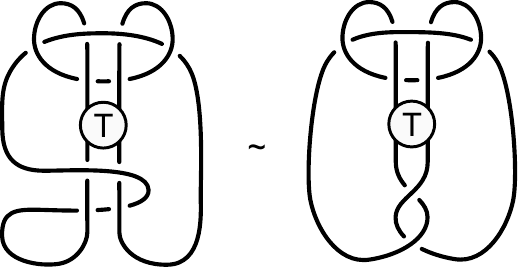}
  \caption{Symmetrising the ribbon $\mathcal{A}(S,T)$ in the case $S=0$}
  \label{fig:ribbonTransformation_S0}
\end{figure} 

%\noindent
The observation about mutants can be generalised (if an unoriented knot diagram 
is obtained from $D$ by switching all crossings we denote it by $D^*$):
\begin{proposition}\label{mutants}
The knots $\mathcal{A}(-S,-T)$ and $\mathcal{A}(S-1,T)^*$ are mutants.
\end{proposition}

\begin{proof}
Let $S \ge 1, T \ge 0$. Start with $\mathcal{A}(-S,-T)$ and mirror the diagram shown on the right 
in Figure \ref{fig:ribbonTransformation} in the plane. This almost yields $\mathcal{A}(S,T)$:
the upper part is exactly as in $\mathcal{A}(S,T)$ if it is mutated by a vertical flip. From the tangle with 
$S$ crossings in the lower part one half-twist is needed to put the overcrossing arc in the right 
position. This results in $\mathcal{A}(S-1,T)^*$. 
\end{proof}

In addition to the mutant pairs (11n73, 11n74) and (12n56, 12n57) this explains 
the pairs (13n611, 13n612), (13n835, 13n973) and (13n1313, 13n1316). 
For $T=\pm 1$, however, the mutation seems to preserve the knot types.
Whether this is true for all $S$ is an open question.  

We have the following conjecture about this family. It is supported also by the knots with 13 crossings in the table. 

\begin{conjecture}\label{conj:det9}
With the exception of 12n268 and the knots with $S=-1, 0$, 
the knots $\mathcal{A}(S,T)$ do not have symmetric union presentations.
\end{conjecture}

\begin{figure}[hbtp]
  \centering
  \includegraphics[scale=\scaling]{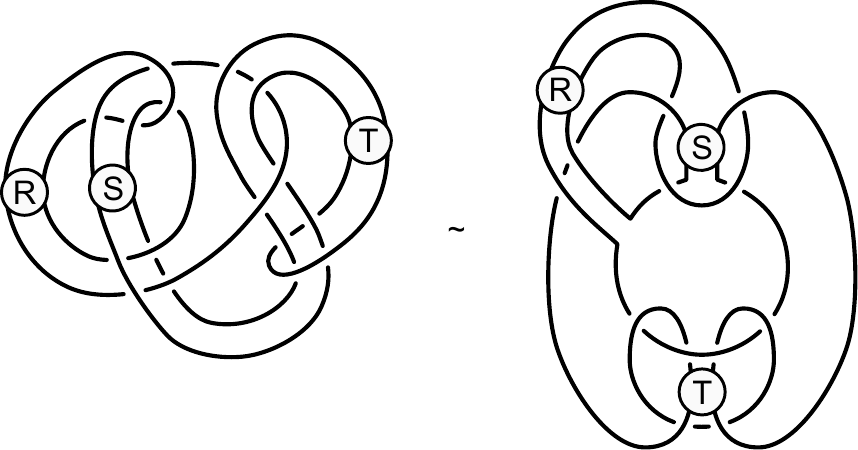}
  \caption{The ribbon $\mathcal{B}(R,S,T)$ for knots with determinant 81}
  \label{fig:ribbon11a103}
\end{figure} 

%%%%%%%%%%%%%%%%%%%%%%%%%%%%%%%%%%%%%%%%%%%%%%%%%%%%%%%
\subsection{A non-symmetric ribbon for the remaining ribbon knots with determinant 81} \label{sub:det81}

For the 5 knots with determinant 81 (i.e.\,11a103, 11a165, 11a201, 12n62 and 12n66) 
we found the ribbon $\mathcal{B}$ shown in Figure \ref{fig:ribbon11a103}. A transformation 
to a diagram with more symmetry is possible as well and we also detect a `symmetry defect'.
(By this we mean the mixed occurence of symmetry and its violation, similar to the case in 
Figure \ref{fig:saddle_transformation_12a1225}. Proposition \ref{prop:Aceto} contains a 
rigorous characterisation of non-symmetric bands.)

Varying the twist parameters $R$, $S$ and $T$ we obtain the knots in Table \ref{tab:ribbon11a103}.
The left column contains the 5 knots with 11 and 12 crossings which are candidates for 
non-symmetric ribbon knots. In the middle column we show, in addition, knots with 13 crossings 
for which no symmetric presentation is known and the right column contains ribbon 
knots with known symmetric union presentations.

% belongs to the following table
\renewcommand{\arraystretch}{1.1}

\noindent
\begin{table}[htbp]
\begin{tabular}{|r|r|r|l||r|r|r|l||r|r|r|l|}
\hline
R & S & T& knot & R & S & T& knot & R & S & T& knot \\
\hline
 -1 & 0 & -1 &                                                            &  -2 & 0 & -1 & 13n403  &  1 &-1 & -1 & $10_{99}$ \\
  0 & 1 &  0 &\raisebox{1.5ex}[0pt]{$\bigg\}$11a103}  &   2 & 0 & -1 & 13n405  & -1 &0 & 0  &  11a58 \\
  0 & -1 & -1 &                                                           &   1 & 1 &  0 & 13n415  &  1 & 0 & -1 &   12n49 \\
 -1 & 2 &  0 & \raisebox{1.5ex}[0pt]{$\bigg\}$11a165} &   2 & 0 &  0 & 13n440  & 1 &-1 & 0 &  12n440 \\
  0 &-1 &  0 & \phantom{x} 11a201&  0 &2 &  0 & 13n1194 &  -2 & 0 &  0  & 13n436\\
  1 & 0 &  0 & \phantom{x} 12n62 &  -1&-1 & -1 & 13n1198 &   1 &  1 & -1 & 13n447\\
 -2 & 1 & -1 & \phantom{x} 12n66 &   1 &  0 & 1 & 13n1336 & -1 & -1 & 0  & 13n1156\\
&&&                         &  -1 &  0 & 1 & 13n1680 &  1 & -1 & 1  & 13n2199\\
 0 & 0 & -1 & \phantom{x} $3_1\#8_{11}$ &   0 & 1 &  1 &  13n1971 &&&&\\
 0 & 0 &  0 & \phantom{x} $3_1\#8_{10}$ &   0 & -1 & 1 &  13n3058 &&&&\\
\hline
\end{tabular}
%\bigskip
\caption{Overview of knot types for the ribbon $\mathcal{B}(R,S,T)$ with determinant 81. 
Knots in the right column are symmetric.}
\label{tab:ribbon11a103}
\end{table}
% belongs to the previous table
\renewcommand{\arraystretch}{1.0}

We interpret the fact that this family contains symmetric and possibly non-symmetric ribbon knots
in the following way: the knots in the table cannot be transformed to a symmetric union in a way that their 
ribbon is respected. However, for some of them there exists a symmetric union presentation with another ribbon.
Put another way, the knots in the right block in Table \ref{tab:ribbon11a103} correspond to the case 12n268 in 
the previous section.

\bigskip
\noindent
For $R=S=0$ the ribbon knots are composite and we obtain for $T=-3,\ldots,3$ the knots
$3_1\#11n122$, $3_1\#10_{143}$, $3_1\#8_{11}$, $3_1\#8_{10}$, $3_1\#10_{147}$, $3_1\#11n106$, $3_1\#12n444$.

%\medskip
%\renewcommand{\arraystretch}{1.1}
%\begin{tabular}{|c|c|c|c|c|c|c|}
%\hline
%$-3$ & $-2$ & $-1$ & 0 & 1 & 2 & 3 \\
%\hline
%$3_1\#11n122$ & $3_1\#10_{143}$ & $3_1\#8_{11}$ & $3_1\#8_{10}$ & $3_1\#10_{147}$ & $3_1\#11n106$ & $3_1\#12n444$ \\
%\hline
%\end{tabular}
%\renewcommand{\arraystretch}{1.0}

\medskip
These composite ribbon knots consist of a trefoil and a knot concordant to the mirrored trefoil.
Because of their striking non-symmetry, we think that they do not have symmetric union 
presentations and that proving this might be easiest among our candidates.

\begin{conjecture}\label{conj:det81}
The (composite) knots $\mathcal{B}(0,0,T)$ are non-symmetric ribbon knots;
less confident than for the determinant 9 cases: the knots 11a103, 11a165, 11a201, 
12n62 and 12n66 are non-symmetric ribbon knots. And, with specific ribbon: 
no knot $\mathcal{B}(R,S,T)$ can be transformed to a symmetric union respecting its ribbon. 
\end{conjecture}

\newpage
%%%%%%%%%%%%%%%%%%%%%%%%%%%%%%%%%%%%%%%%%%%%%%%%%%%%%%%
\subsection{Non-symmetric ribbons for the remaining ribbon knots with determinant 225} \label{sub:det225}
For the three knots 12a348, 12a990 and 12a1225 we found the two ribbons shown in Figure \ref{fig:ribbons_C_D}. 
Ribbon $\mathcal{C}$ yields for parameters $(S,T)=(0,1)$ the knot 12a990 and for $(2,0)$ the knot 12a1225. 
No knots with 13 crossings have been found for other parameters.

The knot 12a348 is obtained for the ribbon $\mathcal{D}$ with parameters $(-2,1,-1)$ and $(1,-2,0)$ and additional
ribbon knots with 13 crossings are reported in Table \ref{tab:ribbon12a348}. For none of these a symmetric
union presentation has been found.

% belongs to the following table
\renewcommand{\arraystretch}{1.1}

\noindent
\begin{table}[htbp]
\begin{tabular}{|r|r|r|l||r|r|r|l|}
\hline
R & S & T& knot & R & S & T& knot \\
\hline
 -2 & 1 & -1 &                                                             &  -1 & -1 &  0 & \phantom{x} 13a549 \\
  1 & -2 &  0 &\raisebox{1.5ex}[0pt]{$\bigg\}$12a348}  &  0 & 0 & -1 & \phantom{x} 13a556   \\
  & &  &                                                                      &  -2 & 0 &  0 & \\
  -2 & 0 & -1 & \phantom{x} 13a414                             & 1 & -1 & -1 &\raisebox{1.5ex}[0pt]{$\bigg\}$13a907}   \\
  0 & 0 & 0 &                                                              &  1 &-1 & 0 & \phantom{x} 13a1392 \\
  -1 & -1 & -1 &\raisebox{1.5ex}[0pt]{$\bigg\}$13a521} & & & &  \\
\hline
\end{tabular}
%\bigskip
\caption{Overview of knot types for the ribbon $\mathcal{D}(R,S,T)$.
}
\label{tab:ribbon12a348}
\end{table}
% belongs to the previous table
\renewcommand{\arraystretch}{1.0}

It seems plausible to conjecture that the ribbons $\mathcal{C}$ and $\mathcal{D}$ generate only non-symmetric 
ribbon knots. Effective tools for tackling our conjectures do not yet exist, though.

\begin{figure}[hbtp]
  \centering
  {\includegraphics[scale=0.7]{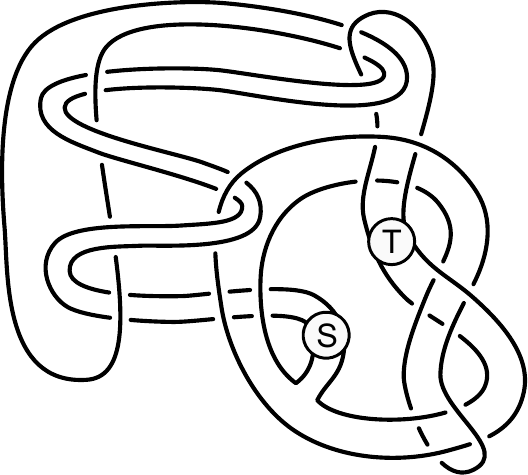}\qquad}
	{\includegraphics[scale=0.7]{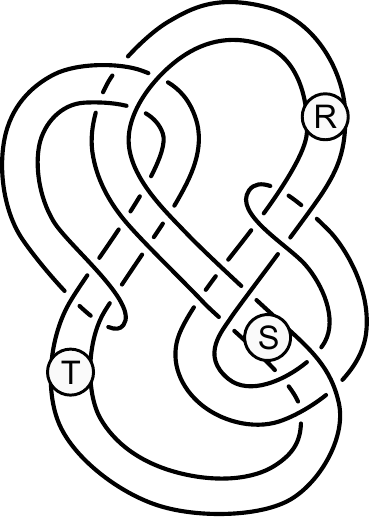}}
  \caption{The ribbons $\mathcal{C}(S,T)$ and $\mathcal{D}(R,S,T)$ for knots with determinant 225}
  \label{fig:ribbons_C_D}
\end{figure} 

%%%%%%%%%%%%%%%%%%%%%%%%%%%%%%%%%%%%%%%%%%%%%%%%%%%%%%%

\clearpage
\newpage
\section{Symmetric union diagrams as bands without crossings and junctions} \label{sec:banddiagrams}
In this section we characterise bands constructed from symmetric unions as bands composed 
of a reduced set of elementary pieces. First, we introduce a normalised form for ribbon bands: 
Every ribbon band can be represented by the elementary pieces of Figure \ref{fig:elementarypieces}, 
see \cite{Eisermann}. 
If we restrict this list of building blocks by omitting junctions and crossings, we obtain symmetric bands:

\begin{figure}[hbtp]
 {\includegraphics[scale=0.8]{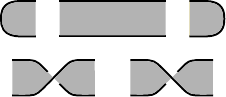}\quad}
 {\qquad\includegraphics[scale=0.8]{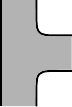}\quad}
 {\quad\includegraphics[scale=0.8]{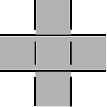}\quad}
 {\qquad\includegraphics[scale=0.8]{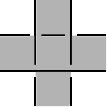}}
 \caption{Elementary pieces of band diagrams: ends, strip, twists, junction, crossing, ribbon singularity}
 \label{fig:elementarypieces}
\end{figure}

\begin{proposition}[\cite{Aceto}]\label{prop:Aceto}
A band can be constructed from a symmetric union diagram if and only if it has a band 
diagram composed of elementary pieces with neither crossings nor junctions.
\end{proposition}

One direction in the proof of this proposition is clear from an illustration as in Figure \ref{fig:achter_3D}:
a slight modification of a symmetric band's projection on the symmetry plane has neither crossings
nor junctions and only strips, ends, ribbon singularities and twists occur. Aceto observed that
the other direction is true as well and obtained an estimate concerning the number of ribbon
singularities in a specific family of ribbon knots.

We noticed earlier (see Figure \ref{fig:ribbonTransformation_S0}) that for $S=-1,0$ the bands
$\mathcal{A}(S,T)$ can be transformed to a symmetric band. Taking Proposition \ref{prop:Aceto}
into account, the symmetrisation of bands can be viewed as an elimination of crossings and junctions.
For this, we provide two band moves in the following example.

\begin{figure}[hbtp]
  \centering
  \includegraphics[scale=0.7]{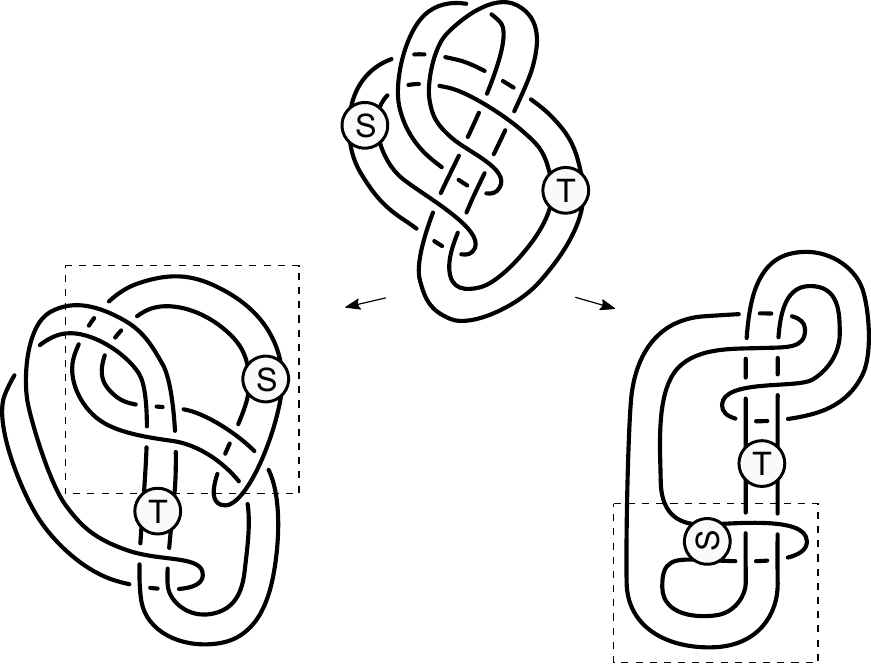}
  \caption{Preparation for the elimination of a crossing (left) or a junction (right)}
  \label{fig:prep_vanishing}
\end{figure}

\begin{figure}[hbtp]
  \centering
  \includegraphics[scale=\scaling]{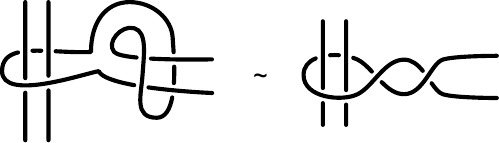}
  \caption{A band move useful for the elimination of a crossing}
  \label{fig:new_move}
\end{figure} 

\begin{figure}[hbtp]
  \centering
  \includegraphics[scale=\scaling]{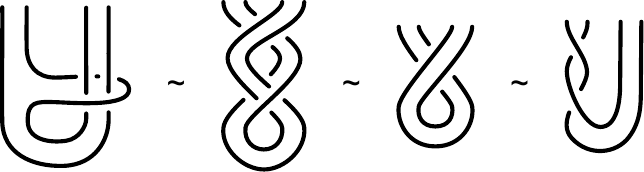}
  \caption{Removing a junction in $\mathcal{A}(-1,T)$ using Reidemeister 1 moves for bands}
  \label{fig:junction_move}
\end{figure}

\begin{example}
We illustrate this again with the ribbon $\mathcal{A}(S,T)$. 
As the case $S = 0$ is already shown in Figure \ref{fig:ribbonTransformation_S0},  we now choose the case $S = -1$.

Elimination of a crossing: In the left side of Figure \ref{fig:prep_vanishing} the diagram is 
modified so that for $S=-1,0$ the crossing in the upper part can be removed. 
The move which eliminates this crossing for $S=-1$ is shown in Figure \ref{fig:new_move}.

Elimination of a junction: In the right side of Figure \ref{fig:prep_vanishing} we transform the diagram in a similar
way as in Figure \ref{fig:ribbonTransformation_S0} and prepare the elimination of a junction. 
Reidemeister 1 type moves (with and without singularity) allow the removal of the junction, 
see Figure \ref{fig:junction_move}.

Note that these eliminations are not possible for $S \neq -1,0$.
\end{example}

Why, then, are the ribbons $\mathcal{B}(R,S,T)$ in Figure \ref{fig:ribbon11a103} not symmetric 
although they seem to be free of crossings and junctions? This is explained in Figure \ref{fig:auxiliary_piece}:
for convenience we allow an additional auxiliary elementary piece in the band construction. 
If this is expressed by elementary pieces, a junction or crossing is introduced, showing that
the bands $\mathcal{B}(R,S,T)$ are not symmetric.

\begin{figure}[hbtp]
  \centering
  \includegraphics[scale=\scaling]{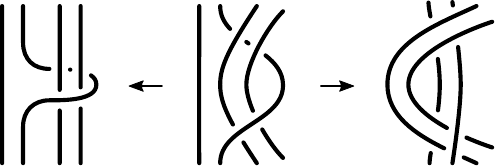}
  \caption{The auxiliary elementary piece in the middle transforms to a description with 
	either a junction or a crossing if only elementary pieces are to be used.}
  \label{fig:auxiliary_piece}
\end{figure} 

\begin{figure}[hbtp]
  \centering
  \includegraphics[scale=0.7]{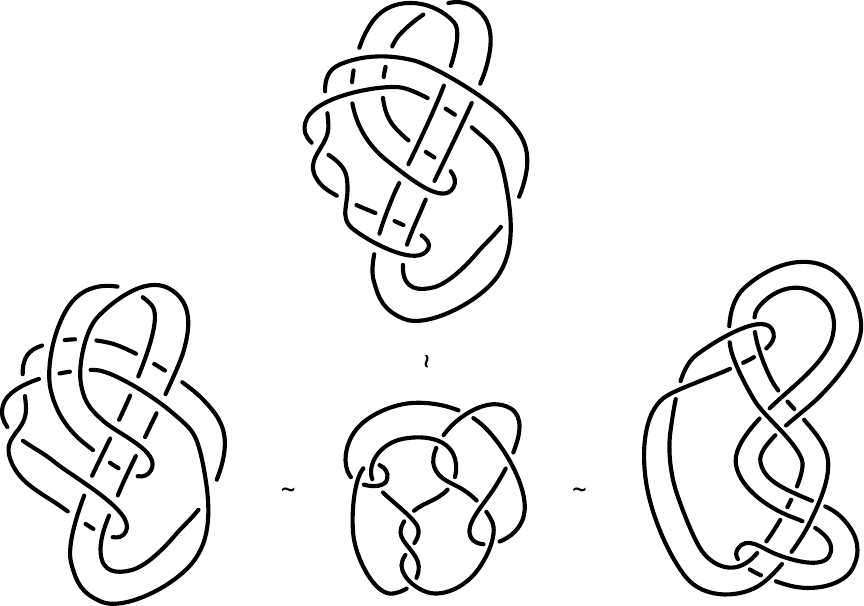}
  \caption{Three ribbons representing the knot 12n268 (shown in the middle): are the band diagrams $\mathcal{A}(2,-1)$ 
   (left), $\mathcal{A}(-3,1)^*$ (top), and the diagram taken from template $3_1$, version b (right) equivalent?
	 The left and top diagrams are related by mutation (Proposition \ref{mutants}) and the right diagram is symmetric.}
  \label{fig:12n268_bands}
\end{figure} 

The next step, a systematic study of band surfaces, their equivalence relations and invariants, 
is postponed to a future article. Suitable tools should then allow one, for instance, to show the equivalence 
or non-equivalence of the bands in Figure \ref{fig:12n268_bands}, all representing the knot 12n268.
An analogous investigation for symmetric ribbons has been done in our articles [6] and [7].

\newpage
\section{Knotoid notation for symmetric ribbon disks}\label{knotoids}
This section contains the description of a new diagram type for bands constructed 
from symmetric unions. It can be seen as a continuation of Aceto's proposition in 
the last section and was already used in Figure \ref{fig:achter_3D} in the insets.

This notation is called \textsl{knotoid diagram for a symmetric ribbon disk} and is motivated 
by the observation that one half of a symmetric union diagram is redundant and that the 
height function, necessary to construct a band from a symmetric union diagram, can be 
taken directly from the diagram. This facilitates the construction and visualisation of symmetric bands. 

\begin{figure}[hbtp]
  \centering
  \includegraphics[scale=\scaling]{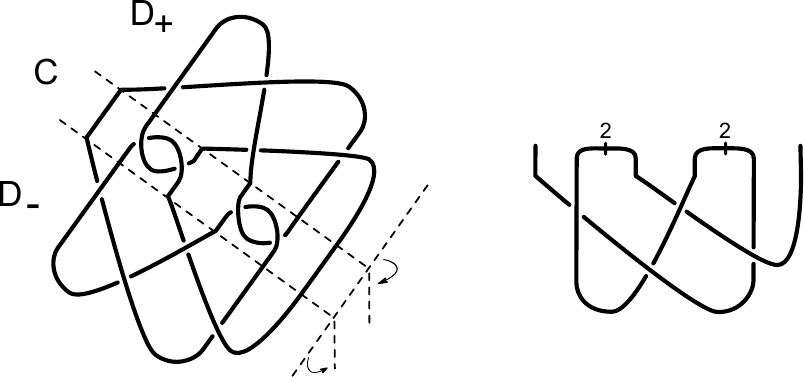}
  \caption{Illustration of the knotoid definition}
  \label{fig:knotoid_def}
\end{figure} 

\begin{definition}[knotoid diagram for a symmetric union]\label{knotoiddef}
Consider a symmetric union diagram $D$, consisting of the two mirror-symmetric parts 
$D_-$ and $D_+$ and the central part $C$ as in Figure \ref{fig:knotoid_def}. 
The knotoid diagram for $D$ is the arc obtained from $D_-$ by identifying the ends after 
cutting open the crossings on the axis. The (signed) numbers of crossings at the twist tangles in $C$ 
are used to decorate the arc so that the diagram $D$ can be recovered from its knotoid diagram. 
The knotoid endpoints correspond to the two perpendicular crossing points of the diagram with the axis. 
By folding the diagram down as indicated in the figure -- using the diagram to 
supply the height function -- we define the \textsl{knotoid band} given by the knotoid diagram of D. 
It is a ribbon of constant width except in the regions near crossings and is displayed in Figure \ref{fig:achter_3D}.
\end{definition}

This is easily extended to the notion of \textsl{knotoid diagram for a symmetric ribbon disk}:
We allow decorations with twist numbers for every segment between the arc's crossings and 
drop the requirement for the endpoint positions. 
%In this more general situation the corresponding knotoid band is constructed as in Definition \ref{knotoiddef}. 
In more detail, a knotoid diagram is a marked generically embedded arc $A$ 
in the $(x,y)$-plane. It determines a `vertical' immersion of $A \times [-1,1]$ in $\R^3$.
The markings indicate how the immersion can be perturbed into a ribbon immersion.
This step uses the over/under crossing information for choosing the ribbon singularities and the twist
decorations for inserting twists. Note that it is not possible to construct a ribbon disk containing band 
crossings or junctions from a knotoid diagram.

\begin{figure}[hbtp]
  \centering
  \includegraphics[scale=0.6]{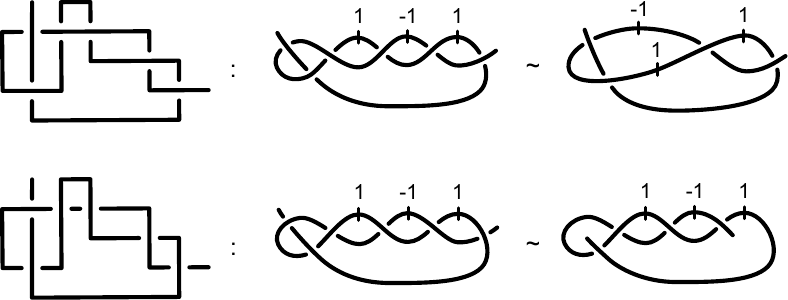}
  \caption{Knotoid diagrams for template $3_1$, version b. The second and third columns show the knot 12n268.}
  \label{fig:arc_diagrams}
\end{figure}

We chose the name due to the similarity to knotoids due to Vladimir Turaev \cite{Turaev}. 
Compare also the related definition of arc diagrams for ribbon 2-knots\footnote{In the context of 2-dimensional 
knots the notion of `symmetric ribbon spheres' was used by Takeshi Yajima \cite{Yajima}. He proved that 
ribbon 2-knots can be deformed to a symmetric position in $\R^4$ (symmetric to the plane $x_4=0$). For comparison 
with classical knots we summarise the situation for 2-knots: all 2-knots are slice. There are slice 2-knots which 
are not ribbon. All ribbon 2-knots are symmetric (in the above sense).}, see \cite{KanenobuKomatsu}. 
Knotoid diagrams and -bands enjoy the following properties: 
\begin{itemize}
\item[a)]	Twist decorations on the same underpassing arc can be added because ribbon 
twists can be moved through singularities. This is not possible for overpassing arcs 
because the movement is prevented by the intersecting band.
\item[b)]	Reidemeister moves for the knotoid diagram correspond to symmetric Reidemeister moves 
of the symmetric union diagram. They might be restricted by the arc's twist decorations. 
Analogously to the knotoid (and arc) diagram case the end of an arc can only be pulled 
out if it is an underpassing end. For an illustration see Figure \ref{fig:arc_diagrams}.
\item[c)]	The insertion of twists in the band is easier for knotoid diagrams than for symmetric union diagrams: 
for the latter for arcs far away from the axis it is necessary to over- or undercross parts of the diagram. 
Knotoid diagrams avoid this because twists can be inserted everywhere.
\end{itemize}

%\enlargethispage{0.6cm}

Interestingly, knotoid bands for a knot given by a symmetric union diagram exist in two variants: 
if the diagram is rotated around its axis the diagram parts $D_-$ and $D_+$  are exchanged and the knotoid 
corresponding to $D_+$ after rotation may yield a different band for the same knot, see Figure \ref{fig:folding}
(to use the rotated $D_+$ and fold up is the same as to use $D_-$ and fold down). 
Figure \ref{fig:arc_diagrams} contains two surprisingly different simplifications of the two bands for the knot 
12n268 obtained in this way.

In Figure \ref{fig:knotoid-practical} we explain the practical construction of a band from a knotoid diagram. 
In the same way a band for a complicated symmetric union can be constructed -- via its knotoid diagram.
We propose to study knotoid diagrams for symmetric ribbon disks in detail in a separate article.

\begin{figure}[hbtp]
  \centering
  \includegraphics[scale=\scaling]{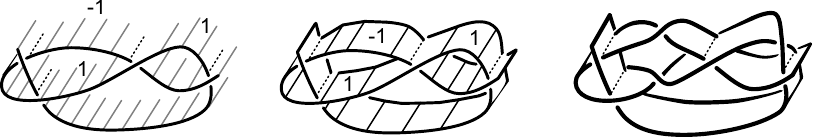}
  \caption{Practical steps for constructing the band from a knotoid diagram}
  \label{fig:knotoid-practical}
\end{figure}

\begin{figure}[hbtp]
  \centering
  \includegraphics[scale=0.5]{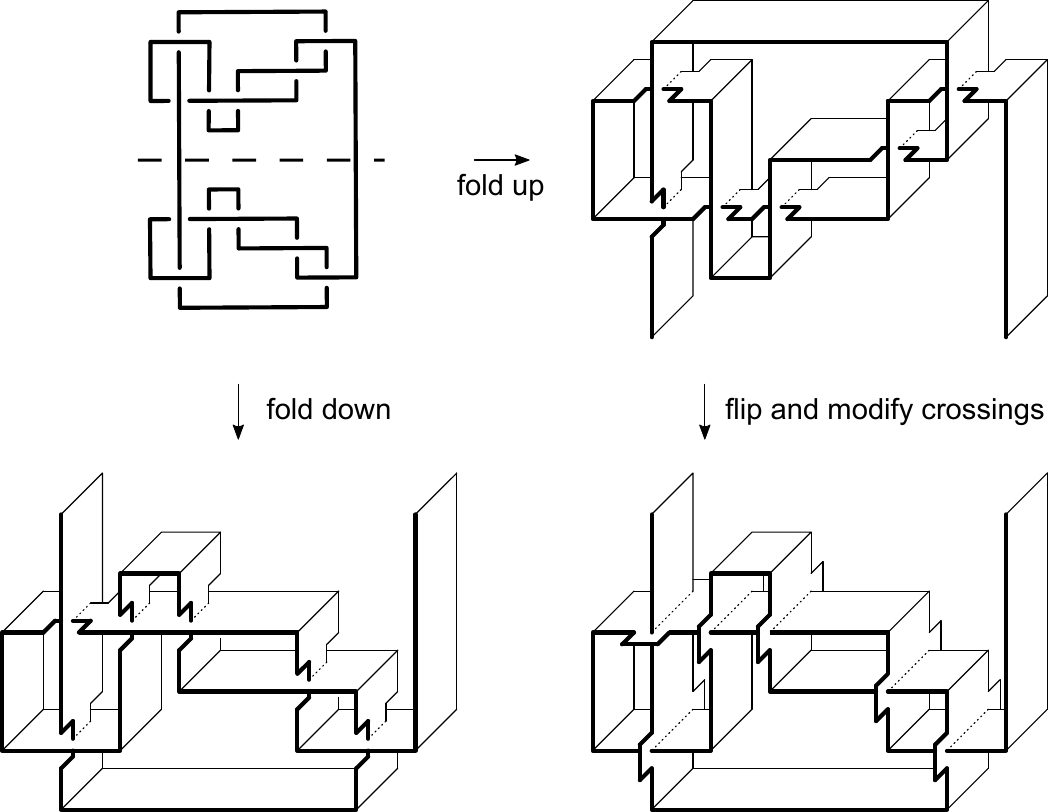}
  \caption{The two folding operations, illustrated with template $3_1$, version b}
  \label{fig:folding}
\end{figure}

%\clearpage
%\newpage
\section*{Outlook}\label{outlook}
Besides proving that non-symmetric ribbon knots exist and pursuing the topic of band equivalences and
corresponding invariants, an attractive goal is the study of ribbon knots with 13 crossings. For them Seeliger 
found 246 prime symmetric ribbon knots and an interesting discrepancy in the identification rate for different
determinants: for instance, for determinant 289 there are 35 candidates and 32 symmetric hits, while for 
determinant 441 there are 23 candidates and no hits. 
%The total number the search found for prime symmetric ribbon knots with 14 crossings is 721.

In another direction, a question similar to the slice-ribbon problem occurs: if we define symmetric slice 
knots by requiring a symmetric disk and allow twists at the saddles on the symmetry plane, can we prove 
that they are symmetric ribbon knots, i.e.\,can local minima be avoided?

\section*{Acknowledgements}
I thank Lukas Lewark and Johannes Renkl for commenting on earlier versions of this article and the referee
for additional valuable help. 
% ... supported this analysis of his results and provided answers to several detailed questions.
The support of Axel Seeliger, who has left knot theory research to work in industry, is gratefully acknowledged.

%\enlargethispage{0.5cm}

%\clearpage
%\newpage
%%%%%%%%%%%%%%%%%%%%%%%%%%%%%%%%%%%%%%%%%%%%%%%%%%%%%%%

\vspace{1cm}
\noindent
Christoph Lamm \\ \noindent
R\"{u}ckertstr. 3, 65187 Wiesbaden \\ \noindent
Germany \\ \noindent
e-mail: christoph.lamm@web.de

\clearpage
%%%%%%%%%%%%%%%%%%%%%%%%%%%%%%%%%%%%%%%%%%%%%%%%%%%%%%%
% Appendix: diagrams
%%%%%%%%%%%%%%%%%%%%%%%%%%%%%%%%%%%%%%%%%%%%%%%%%%%%%%%
\section{Appendix: the 28 templates generating 122 symmetric union diagrams}
\noindent
The diagrams are grouped by determinant; further explanations are given in the text.

\vspace{1cm}

% all text in the diagram overview should be \small except the info on determinants
\small
%
%%% det = 1 %%%%%%%%%%%%%%%%%%%%%%%%%%%%%%%%%%%%%%%%%%%%%%%%%%%%%%%

{\large
\marginpar{\hfill\colorbox{myboxcolour}{$\det=1$}}
}

% template triv_5cr
\noindent
\parbox[t]{3.0cm}{
\centering
\mbox{} \\
\includegraphics[scale=\scaling]{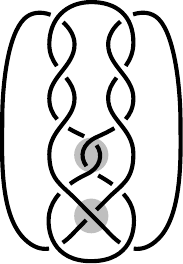} \\
$K_{\pm} =$ triv, version a
}
% instances for template triv_5cr
\parbox[t]{3.2cm}{
\centering
\mbox{} \\
\begin{tabular}{lcr@{, }r}
11n42   &=&(2& 0) \\
11n49   &=&(1& -1) \\
11n116  &=&(1&  1) \\
12n214  &=&(3&  0) \\
12n309  &=&(1& -2) \\
\underline{12n313}  &=&(2& -1) \\
12n318  &=&(1&  2) \\
12n430  &=&(2&  1)
\end{tabular}
}
\quad
% template triv_6cr 
\parbox[t]{4cm}{
\centering
\mbox{} \\
\includegraphics[scale=\scaling]{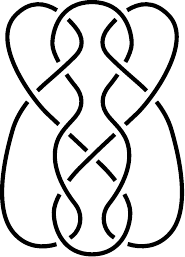} \\
$K_{\pm} =$ triv, version b \\
12n19 = (1)
}

\vspace{1cm}

%%% det = 9 %%%%%%%%%%%%%%%%%%%%%%%%%%%%%%%%%%%%%%%%%%%%%%%%%%%%
{\large
\marginpar{\hfill\colorbox{myboxcolour}{$\det=9$}}
}

% template 3_1
\noindent
\parbox[t]{3.5cm}{
\centering
\includegraphics[scale=\scaling]{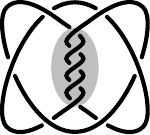} \\
$K_{\pm} = 3_1$, version a \\
\underline{11n139} = (5) \\
12n582 = (6) \\
\mbox{}
}
\quad
% template 3_1_6cr
\noindent
\parbox[t]{4cm}{
\centering
\includegraphics[scale=\scaling]{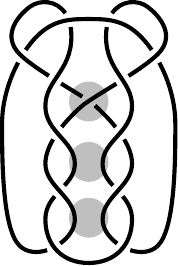} \\
$K_{\pm} = 3_1$, version b \\
\begin{tabular}{lcr@{, }r@{, }r}
12n268  &=&(1& -1&  1) \\
\underline{12n605}  &=&(1&  0&  0) \\
\end{tabular}
\mbox{}
}
\quad
\noindent
\parbox[t]{4cm}{
\centering
\includegraphics[scale=\scaling]{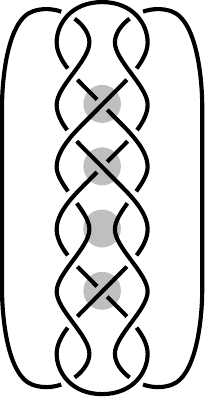} \\
$K_{\pm} = 3_1$, version c \\
\begin{tabular}{lcr@{, }r@{, }r@{, }r}
\underline{12n23}  &=&(1& -1&  0& 1) \\
12n676 &=&(0&  1& -3& 1) \\
\end{tabular}
\mbox{}
}

\vspace{1cm}

%%% det = 25 %%%%%%%%%%%%%%%%%%%%%%%%%%%%%%%%%%%%%%%%%%%%%%%%%%%
{\large
\marginpar{\hfill\colorbox{myboxcolour}{$\det=25$}}
}

% templates 4_1 and 5_1
\noindent
\parbox[t]{2.5cm}{
\centering
\mbox{} \\
\includegraphics[scale=\scaling]{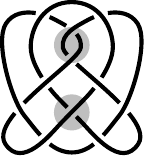} \\
$K_{\pm} = 4_1$ 
}
% instances for template 4_1
\parbox[t]{3.3cm}{
\centering
\mbox{} \\
\begin{tabular}{lcr@{, }r}
11n50   &=&(3&  0) \\
\underline{11n132}  &=&(2& 1) \\
12n145  &=&(4&  0) \\
12n462  &=&(2& -2) \\
12n768  &=&(3&  1) \\
12n838  &=&(2&  2) 
\end{tabular}
}
\quad
\parbox[t]{3.3cm}{
\centering
\mbox{} \\
\includegraphics[scale=\scaling]{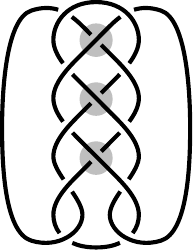} \\
$K_{\pm} = 5_1$ 
}
% instances for template 5_1
\parbox[t]{3.7cm}{
\centering
\mbox{} \\
\begin{tabular}{lcr@{, }r@{, }r}
11n37   &=&(0&  1&  0) \\
\underline{11n39}  &=&(1& -1& -1) \\
12n256  &=&(1& -1& -2) \\
12n257  &=&(1& -1&  2) \\
12n279  &=&(0&  1&  1) \\
12n394  &=&(2& -1& -1) \\
12n414  &=&(1&  1&  0) \\
12n670  &=&(0&  2&  0) \\
12n721  &=&(2&  0&  0) \\
12n870  &=&(2& -1&  2) \\
\end{tabular}
}

\vfill

\begin{center}
List of symmetric diagrams, page 1 (determinants 1, 9 and 25)
\end{center}

%\vspace{1cm}
\newpage

%%% det = 49 %%%%%%%%%%%%%%%%%%%%%%%%%%%%%%%%%%%%%%%%%%%%%%%%%%%%%%%
{\large
\marginpar{\hfill\colorbox{myboxcolour}{$\det=49$}}
}

% template 5_2
\noindent
\parbox[t]{3.0cm}{
\centering
\mbox{} \\
\includegraphics[scale=\scaling]{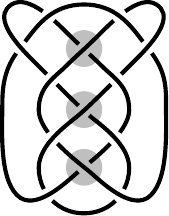} \\
$K_{\pm} = 5_2$ 
}
% instances for template 5_2
\parbox[t]{3.9cm}{
\centering
\mbox{} \\
\begin{tabular}{lcr@{, }r@{, }r}
\underline{11n4} &=&(1& 1& -1) \\
11n21   &=&(0&  1&  0) \\
11n83   &=&(1& -1&  0) \\
11n172  &=&(1& -1&  3) \\
12n24   &=&(1& -1& -1) \\
12n48   &=&(2&  0& -1) \\
12n87   &=&(0&  2&  0) \\
12n288  &=&(0&  0&  2) \\
12n312  &=&(0&  1&  1) \\
12n360  &=&(1&  1&  0) \\
12n393  &=&(1& -1&  4) \\
12n397  &=&(2&  1& -1) \\
12n501  &=&(2&  0&  0) \\
12n708  &=&(1&  0&  1) \\
12n817  &=&(1&  0& -2)
\end{tabular}
}
\quad\quad
% template 7_1
\parbox[t]{4cm}{
\centering
\mbox{} \\
\includegraphics[scale=\scaling]{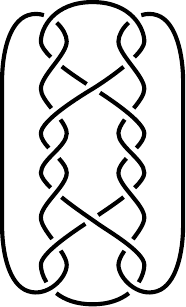} \\
$K_{\pm} = 7_1$ \\
12n706 = (1, -1)
}

\vspace{1cm}

%%% det = 81 %%%%%%%%%%%%%%%%%%%%%%%%%%%%%%%%%%%%%%%%%%%%%%%%%%%%%%%
{\large
\marginpar{\hfill\colorbox{myboxcolour}{$\det=81$}}
}

\noindent
% template 6_1
\parbox[t]{3.2cm}{
  \centering
  \mbox{} \\
  \includegraphics[scale=\scaling]{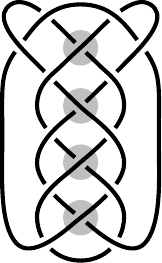} \\
  $K_{\pm} = 6_1$, version a 
}
% instances for template 6_1
\parbox[t]{4.2cm}{
  \centering
  \mbox{} \\
  \begin{tabular}{lcr@{, }r@{, }r@{, }r}
      12a425  &=&(0&  0&  0&  1) \\
      12a1029 &=&(1&  0&  0&  0) \\
      12a1283 &=&(1&  0&  0& -1) \\
      12n4    &=&(0&  1& -1& -1) \\
      12n43   &=&(0&  1&  0& -1) \\
      12n106  &=&(0&  1&  0& -2) \\
      12n170  &=&(1& -1&  0&  1) \\
      12n380  &=&(0&  0&  1& -1) \\
      12n399  &=&(0&  1& -1&  2) \\
      \underline{12n420}  &=&(1& -1&  1&  1) \\
      12n440  &=&(1&  0&  1& -2) \\
      12n636  &=&(2& -1&  1&  0) \\
      12n657  &=&(1&  0& -1&  1) \\
      12n876  &=&(1& -1&  0&  2) 
  \end{tabular}
}
\quad
\parbox[t]{4.7cm}{
\centering
\mbox{} \\
\includegraphics[scale=\scaling]{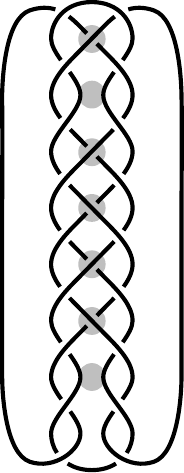} \\
$K_{\pm} = 9_1$ \\
11a58 = (0, 0, 1, -1, 0,  1, -1)\\
\underline{12n49} = (1, 0, 1, -1, 1, -1,  0)
}

\vspace{1cm}

\noindent
% template 6_1_v3
\parbox[t]{3.5cm}{
\centering
\mbox{} \\
\includegraphics[scale=\scaling]{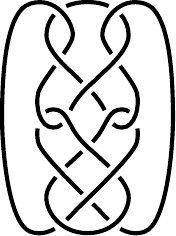} \\
$K_{\pm} = 6_1$, version b \\
12n782 = (1, -1, 1)
}
\quad
% template 3_1#3_1
\parbox[t]{3.8cm}{
\centering
\mbox{} \\
\includegraphics[scale=\scaling]{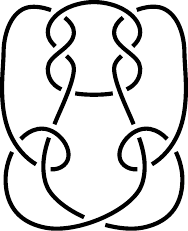} \\
$K_{\pm} = 3_1\#3_1$ \\
12n553 = (1)
}
\quad
% template  3_1#3_1*
\parbox[t]{3.8cm}{
\centering
\mbox{} \\
\includegraphics[scale=\scaling]{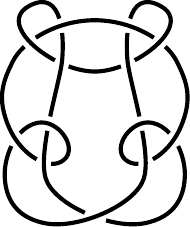} \\
$K_{\pm} = 3_1\#-3_1$ \\
12n556 = (1)
}

\vfill
\begin{center}
List of symmetric diagrams, page 2 (determinants 49 and 81)
\end{center}

\newpage

%%% det = 121 %%%%%%%%%%%%%%%%%%%%%%%%%%%%%%%%%%%%%%%%%%%%%%%%%%%%%%%
{\large
\marginpar{\hfill\colorbox{myboxcolour}{$\det=121$}}
}

\noindent
% template 6_2
\parbox[t]{3.5cm}{
  \centering
  \mbox{} \\
  \includegraphics[scale=\scaling]{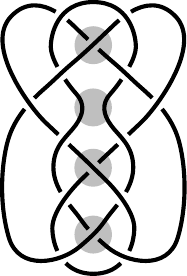} \\
  $K_{\pm} = 6_2$, version a
}
% instances for template 6_2
\parbox[t]{4.2cm}{
  \centering
  \mbox{} \\
  \begin{tabular}{lcr@{, }r@{, }r@{, }r}
      \underline{11a28}   &=&(1&  0& -1&  1) \\
      11a35   &=&(0&  0&  1& -1) \\
      11a36   &=&(0&  1& -1&  1) \\
      11a87   &=&(1& -1&  1& -1) \\
      11a96   &=&(0&  0& -1&  2) \\
      11a115  &=&(0&  1& -1&  2) \\
      11a169  &=&(1&  0& -1&  2) \\
      11a316  &=&(1& -1&  1& -2) \\
      12a183  &=&(0&  1&  0& -1) \\
      12a447  &=&(0&  0&  0&  1) \\
      12a667  &=&(0&  1&  0&  0) \\
      12a879  &=&(1& -1&  0&  1) \\
      12a1011 &=&(1&  0&  0& -1) \\
      12a1034 &=&(1&  0&  0&  0) \\
      12a1277 &=&(1& -1&  0&  0) \\
      12n802  &=&(1& -1& -1&  2) 
  \end{tabular}
}
\hspace{1cm}
\parbox[t]{4.8cm}{
\centering
\mbox{} \\
\includegraphics[scale=\scaling]{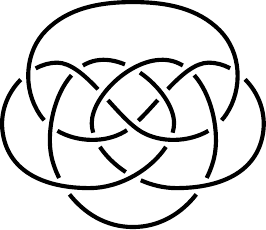} \\
$K_{\pm} = 6_2$, version b \\
12n702 = (1, -1, 1)
}

% positioning not as expected
\vspace{-2cm}{
\hspace{9cm}
\parbox[t]{3.5cm}{
\centering
\mbox{} \\
\includegraphics[scale=\scaling]{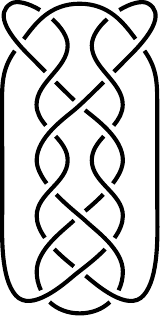} \\
$K_{\pm} = 7_2$ \\
12n504 = (1, 0, -1, 1)
}
}

\vspace{-2cm}

%%% det = 169 %%%%%%%%%%%%%%%%%%%%%%%%%%%%%%%%%%%%%%%%%%%%%%%%%%%%%%%
{\large
\marginpar{\hfill\colorbox{myboxcolour}{$\det=169$}}
}

\noindent
  % template 6_3
  \parbox[t]{3.4cm}{
    \centering
    \mbox{} \\
    \includegraphics[scale=\scaling]{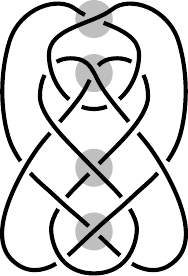} \\
    $K_{\pm} = 6_3$
  }
  % instances for template 6_3
  \parbox[t]{4.0cm}{
    \centering
    \mbox{} \\
    \begin{tabular}{lcr@{, }r@{, }r@{, }r}
      12a54   &=&(1&  0&  0&  0) \\
      12a211  &=&(1& -1& -1&  0) \\
      12a221  &=&(0&  1&  0&  0) \\
      12a258  &=&(1& -1&  0&  0) \\
      12a477  &=&(1&  0& -1&  0) \\
      12a606  &=&(1&  0&  0& -1) \\
      12a819  &=&(0&  1&  0& -1) \\
      12a1119 &=&(1& -1&  0& -1) \\
      \underline{12a1269} &=&(1& -1& -1&  1) 
    \end{tabular}
  }

\vspace{1cm}

\noindent
  % template 7_3
  \parbox[t]{3.0cm}{
    \centering
    \mbox{} \\
    \includegraphics[scale=\scaling]{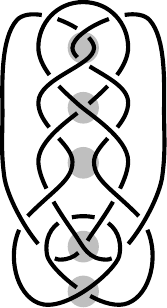} \\
     $K_{\pm} = 7_3$, version a
  }
  % instances for template 7_3
  \parbox[t]{4.6cm}{
    \centering
    \mbox{} \\
    \begin{tabular}{lcr@{, }r@{, }r@{, }r@{, }r}
      11a164  &=&(1& -1&  1&  1& -1) \\
      \underline{11a326}  &=&(2& -1&  0& -1&  1) \\
      12a3    &=&(1& -1&  1& -1&  0) \\
      12a173  &=&(0&  1& -1&  0&  0) \\
      12a279  &=&(1& -1&  1&  0&  0) \\
      12a646  &=&(0&  1& -1&  1&  0) \\
      12a715  &=&(2& -1&  0&  0&  0) \\
      12a786  &=&(2& -1&  0&  1&  0) \\
      12a1083 &=&(3& -1&  0&  0&  0) 
    \end{tabular}
  }
\quad
\parbox[t]{5cm}{
\centering
\mbox{} \\
\includegraphics[scale=\scaling]{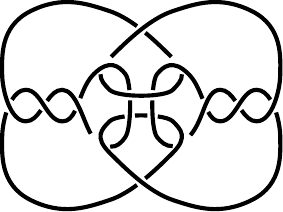} \\
$K_{\pm} = 7_3$, version b \\
12a1202 = (1, -1)
}

\vfill
\begin{center}
List of symmetric diagrams, page 3 (determinants 121 and 169)
\end{center}

\newpage

%%% det = 225 %%%%%%%%%%%%%%%%%%%%%%%%%%%%%%%%%%%%%%%%%%%%%%%%%%%%%%%
{\large
\marginpar{\hfill\colorbox{myboxcolour}{$\det=225$}}
}

\noindent
% template 7_4
\parbox[t]{4.5cm}{
\centering
\mbox{} \\
\includegraphics[scale=\scaling]{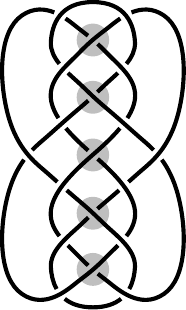} \\
$K_{\pm} = 7_4$ \\
\underline{12a435} = (1, -1, 1, -1,  1)\\
12a464 = (2, -1, 0,  1, -1)\\
12a975 = (2, -1, 1, -1,  2)
}
% template 3_1#4_1
\parbox[t]{4.2cm}{
\centering
\mbox{} \\
\includegraphics[scale=\scaling]{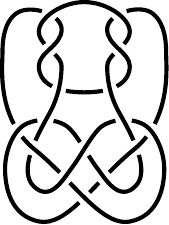} \\
$K_{\pm} = 3_1\#4_1$ \\
12a427 = (1, -1)
}
% template  8_21
\parbox[t]{4.2cm}{
\centering
\mbox{} \\
\includegraphics[scale=\scaling]{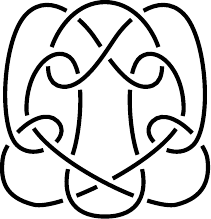} \\
$K_{\pm} = 8_{21}$ \\
12a631 = (1, -1)
}

\vspace{-1.3cm}
\hspace{4.8cm}
% template 9_2
\parbox[t]{2.8cm}{
  \centering
  \mbox{} \\
  \includegraphics[scale=\scaling]{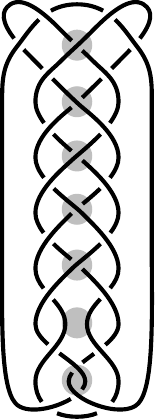} \\
  $K_{\pm} = 9_2$
}
% instances for template 9_2
\parbox[t]{5.6cm}{
  \centering
  \mbox{} \\
  \begin{tabular}{lcr@{, }r@{, }r@{, }r@{, }r@{, }r@{, }r}
    12a77   &=&(0&  1& -1&  0&  1& -1&  0)\\
    12a100  &=&(0&  1& -1&  1& -1&  0&  2)\\
    12a189  &=&(0&  1& -1&  0&  1& -1&  1)\\
    12a245  &=&(1& -1&  1&  0& -1&  1&  0)\\
    12a377  &=&(1& -1&  1&  0& -1&  1& -1)\\
    12a456  &=&(0&  0&  1& -1&  0&  1& -2)\\
    12a979  &=&(1&  0& -1&  1&  0& -1&  2)\\
    \underline{12a1087} &=&(1& -1&  1& -1&  1&  0& -2)
  \end{tabular}
}

\vspace{0cm}

%%% det = 289 %%%%%%%%%%%%%%%%%%%%%%%%%%%%%%%%%%%%%%%%%%%%%%%%%%%%%%%
{\large
\marginpar{\hfill\colorbox{myboxcolour}{$\det=289$}}
}

\noindent
% template 7_5
\parbox[t]{4.0cm}{
\centering
\mbox{} \\
\includegraphics[scale=\scaling]{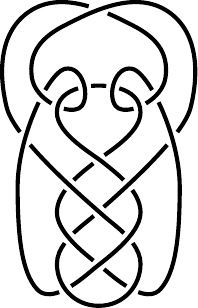} \\
$K_{\pm} = 7_5$, version a \\
12a484 = (1, -1, -1, 1)
}

% template 7_5_v2
\vspace{-4cm}
\hspace{3.9cm}
\parbox[t]{4.5cm}{
\centering
\mbox{} \\
\includegraphics[scale=\scaling]{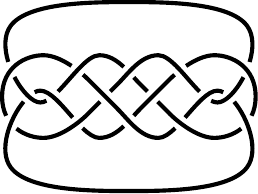} \\
$K_{\pm} = 7_5$, version b \\
12a1105 = (1, -1)
}

% template 8_3
\vspace{-5.7cm}
\hspace{8.5cm}
\parbox[t]{5cm}{
  \centering
  \mbox{} \\
  \includegraphics[scale=\scaling]{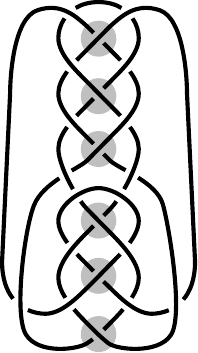} \\
  $K_{\pm} = 8_3$ \\
  \begin{tabular}{lcr@{, }r@{, }r@{, }r@{, }r@{, }r}
    \underline{12a458}  &=&(1& -1&  1& -1&  1& -1)\\
    12a473  &=&(1& -1&  1&  2& -1&  0)\\
    12a887  &=&(2& -1&  0& -2&  1&  0)
  \end{tabular}
}

\vspace{1cm}

%%% det = 361 %%%%%%%%%%%%%%%%%%%%%%%%%%%%%%%%%%%%%%%%%%%%%%%%%%%%%%%
{\large
\marginpar{\hfill\colorbox{myboxcolour}{$\det=361$}}
}

\noindent
% template 7_6
\vspace{-6cm}
\hspace{4.5cm}
\parbox[t]{4.2cm}{
\centering
\mbox{} \\
\includegraphics[scale=\scaling]{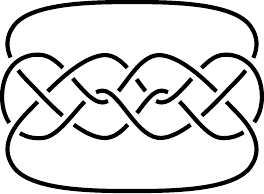} \\
$K_{\pm} = 7_6$ \\
12a1019 = (1, -1)
}

% task: improve the placements in this page
%\enlargethispage{0.5cm}
\vspace{6.1cm}

\begin{center}
List of symmetric diagrams, page 4 (determinants 225, 289 and 361)
\end{center}

\end{document}